\documentclass[11pt,a4paper]{article}
\usepackage{amsmath}
\usepackage{amsfonts}
\usepackage{amssymb}

\usepackage[backend=bibtex,style=alphabetic,backref=true,doi=false,
isbn=false,url=false]{biblatex}
\bibliography{dominating.bib} 

\newtheorem{theorem}{Theorem}

\newtheorem{corollary}[theorem]{Corollary}

\newtheorem{lemma}[theorem]{Lemma}

\newenvironment{proof}[1][Proof]{\noindent\textbf{#1.} }{\ \rule{0.5em}{0.5em}}

\title{Counting Dominating Sets of Graphs}
\author{Irene Heinrich\\
University Kaiserslautern \\
\and Peter Tittmann \\
University of Applied Sciences Mittweida
}

\begin{document}
\maketitle

\begin{abstract}
	Counting dominating sets in a graph $G$ is closely related to the neighborhood
	complex of $G$. We exploit this relation to prove that the number of dominating
	sets $d(G)$ of a graph is determined by the number of complete bipartite subgraphs of
	its complement. More precisely, we state the following. Let $G$ be a simple 
	graph of order $n$ such that its complement $\bar{G}$ has exactly $a(G)$ 
	subgraphs isomorphic to $K_{2p,2q}$ and exactly $b(G)$ subgraphs isomorphic to 
	$K_{2p+1,2q+1}$. Then
	\[
	d(G)=2^n -1 +2[a(G)-b(G)].
	\]
	We also show some new relations between the domination polynomial and the 
	neighborhood polynomial of a graph.
\end{abstract}

\section{Introduction}
Counting dominating sets in graphs offers a multitude of relations to other graphical
enumeration problems. The number of dominating sets in a graph is related to the counting
of bipartite subgraphs, vertex subsets with respect to a given cardinality of the 
intersection of their neighborhoods, 
induced subgraphs such that all their components have even order \cite{Kotek2014}, and 
forests of external activity zero \cite{Dod2015}. In this paper, we show that already 
the number of complete bipartite subgraphs is sufficient to determine the number of dominating
sets of a graph. In addition, we obtain a new proof for the known fact \cite{Brouwer2009} that
the number of dominating sets of any finite graph is odd.

Let $G=(V,E)$ be a simple undirected graph with vertex set $V$ and edge set $E$. The
open neighborhood of a vertex $v$ of $G$ is denoted by $N(v)$ or $N_G(v)$. It is the
set of all vertices of $G$ that are adjacent to $v$. The closed neighborhood of $v$ is
defined by $N_G[v]=N_G(v)\cup \{v\}$. We generalize the neighborhood definitions to
vertex subsets $W\subseteq V$: 
\begin{align*}
	N_G(W) &= \bigcup_{w\in W}N_G(w) \setminus W, \\
	N_G[W] &= N_G(W) \cup W.
\end{align*}
The \emph{edge boundary} $\partial W$ of a vertex subset $W$ of $G$ is
\[ \partial W = \{\{u,v \} \mid u\in W\text{ and }v\in V\setminus W \}, \]
i.e., the set of all edges of $G$ with exactly one end vertex in $W$. Throughout this
paper, we denote by $n$ the number of vertices and by $m$ the number of edges of $G$.

A \emph{dominating set} of $G=(V,E)$ is a vertex subset $W\subseteq V$ with $N[W]=V$.
We denote the number of dominating sets of size $k$ in $G$ by $d_k(G)$. The family of
all dominating sets of $G$ is denoted by $\mathcal{D}(G)$. The
\emph{domination polynomial} $D(G,x)$ is the ordinary generating function for the
number of dominating sets of $G$:
\[ 
	D(G,x) = \sum_{k=0}^{n}d_k(G)x^k = \sum_{W\in \mathcal{D}(G)}x^{|W|} 
\]
 This polynomial has been introduced in \cite{Arocha2000}; it has been further investigated 
 in \cite{Akbari2010, Dod2015, Kotek2012, Kotek2014}.
 The number of dominating sets of $G$ is $d(G)=D(G,1)$.

\section{Alternating Sums of Neighborhood Polynomials}
The \emph{neighborhood complex} $\mathcal{N}(G)$ of $G$ is the family of all subsets
of open neighborhoods of vertices of $G$:
\[ 
	\mathcal{N}(G) =\{X \mid \exists v\in V:X\subseteq N(v) \}  
\]
If $G$ has no isolated vertex, then $\{v\}\in\mathcal{N}(G)$ for any $v\in V$. The
neighborhood complex is a lower set in the Boolean lattice $2^V$, which means that the
relations $X\in\mathcal{N}(G)$ and $Y\subseteq X$ imply $Y\in \mathcal{N}(G)$. Maximal
elements in $\mathcal{N}(G)$ are open neighborhoods of vertices. We define for any 
$k\in\mathbb{N}$, $n_k(G)=|\{X\mid X\in\mathcal{N}(G), |X|=k\}|$. The 
\emph{neighborhood polynomial} of $G$, introduced in \cite{Brown2008}, is
\[ 
	N(G,x) = \sum_{k=0}^{n}n_k(G)x^k = \sum_{W\in \mathcal{N}(G)} x^{|W|}.
\]
If $G'$ is a graph obtained from $G$ by adding some isolated vertices, then 
$\mathcal{N}(G')=\mathcal{N}(G)$ and hence $N(G',x)=N(G,x)$.

\begin{theorem}
	The neighborhood polynomial of any simple graph $G=(V,E)$ satisfies
	\begin{equation*} 
		N(G,x) = \sum_{v\in V}(1+x)^{\deg v} 
		- \sum_{W\subseteq V:|W|\geq 2} (-1)^{|W|} 
		(1+x)^{\left\vert\bigcap_{w\in W}N(w)\right\vert} .
	\end{equation*}
\end{theorem}
\begin{proof}
	First we rewrite the statement of the theorem as
	\begin{equation} 
		N(G,x) = \sum_{\emptyset\neq W\subseteq V} (-1)^{|W|+1} 
		(1+x)^{\left\vert\bigcap_{w\in W}N(w)\right\vert} , 
		\label{eq:nbp-proof}
	\end{equation}
	which is correct as $\deg v = |N(v)|$ for any $v\in V$. The neighborhood polynomial
	is given by
	\begin{equation}
		N(G,x) = \sum_{W\in \mathcal{N}(G)} x^{|W|} 
		= \sum_{W \in \bigcup_{v\in V} 2^{N(v)}} x^{|W|}.
		\label{eq:nbp-proof2}
	\end{equation}
	Comparing Equations (\ref{eq:nbp-proof}) and (\ref{eq:nbp-proof2}), we obtain
	\begin{equation}
		\sum_{W \in \bigcup_{v\in V} 2^{N(v)}} x^{|W|} 
		= \sum_{\emptyset\neq W\subseteq V} (-1)^{|W|+1} 
		(1+x)^{\left\vert\bigcap_{w\in W}N(w)\right\vert}.
		\label{eq:nbp-proof3}
	\end{equation}
	According to the principle of inclusion--exclusion, we have
	\[  \left\vert \bigcup_{v\in V} 2^{N(v)}\right\vert 
	= \sum_{\emptyset\neq W\subseteq V} (-1)^{|W|+1}  
	\left\vert\bigcap_{w\in W}2^{N(w)}\right\vert.
	\]
	Equation (\ref{eq:nbp-proof3}) is just the counting version (or generating function
	expression) of this principle.
\end{proof}
\begin{theorem}\label{theo:null}
	Let $G=(V,E)$ be a simple graph with $E\neq \emptyset$ that is not the disjoint
	union of a complete bipartite graph and an empty (that is edgeless) graph. Then
	\[ \sum_{F\subseteq E} (-1)^{|F|}N(G-F,x) =0. \]
\end{theorem}
\begin{proof}
	Let $W\in \mathcal{N}(G)$ be a vertex subset subset that is contained in the 
	neighborhood complex of $G$. We define the family of edge subsets
	\[ 
		\mathcal{E}(W) =\{F \mid F\subseteq E\text{ and }
		W\in\mathcal{N}(V,F) \}. 
	\]
	Using the definition of the neighborhood polynomial, we obtain
	\begin{align*} 
		\sum_{F\subseteq E} (-1)^{|F|}N(G-F,x) 
		&= \sum_{F\subseteq E} (-1)^{m-|F|}N((V,F),x) \\
		&=  \sum_{F\subseteq E} \; \sum_{W\in \mathcal{N}((V,F))}(-1)^{m-|F|} x^{|W|} \\
		&=  \sum_{W\in \mathcal{N}(G)} x^{|W|} \sum_{F\in \mathcal{E}(W)} (-1)^{m-|F|} 
	\end{align*}
	We will show that the inner sum vanishes, which provides the statement of the theorem. 
	Assume that $E\setminus \partial W\neq \emptyset$. Then for any edge 
	$e\in E\setminus \partial W$ and any edge subset $A\subseteq E$ the relation 
	$A\setminus\{e\}\in \mathcal{E}(W)$ is satisfied if and only if 
	$A\cup \{e\}\in\mathcal{E}(W)$, which yields
	\begin{equation} 
		\sum_{F\in \mathcal{E}(W)} (-1)^{m-|F|}=0. 
		\label{eq:zerosum}
	\end{equation}
	If $E=\partial W$, then $G$ is bipartite and $W$ is one partition set; we denote the 
	second one by $Z$. Since $W\in\mathcal{N}(G)$, there is a vertex $z\in Z$ with $N(z)=W$.
	There exists a vertex $z'\in Z$ with $\emptyset\neq N(z')\subset W$ 
	(\textquotedblleft $ \subset$\textquotedblright meaning proper subset), otherwise $G$ 
	would be a disjoint union of a complete bipartite graph and an empty graph, the latter
	one possibly being the null graph. Now let $e$ be an edge incident to $z'$. Then
	for any set $A\in\mathcal{E}(W)$ that does not contain $e$, the set $A\cup\{e\}$ is also
	in $\mathcal{E}(W)$ and vice versa, $A\cup\{e\}\in\mathcal{E}(W)$ implies
	$A\in\mathcal{E}(W)$.
	The consequence is again that Equation (\ref{eq:zerosum}) is 
	satisfied, which completes the proof.
\end{proof}

\begin{lemma}\label{lemma:pi}
	Let $k,r$ be two positive integers and $\{E_1,\ldots,E_r\}$ a partition of a set $E$ 
	with exactly $r$ blocks of size $k$. Define a family of subsets of $E$ by
	$\mathcal{M}_{kr}=(2^{E_1}\setminus \{E_1\})\times\cdots\times (2^{E_r}\setminus \{E_r\})$.
	Then 
	\[
		\sum_{A\in M_{kr}}(-1)^{|A|} = (-1)^{(k-1)r}.
	\]
\end{lemma}
\begin{proof}
	First consider the case $r=1$. As $E_1\neq \emptyset$ we have
	\[  
		\sum_{A\subseteq E_1}(-1)^{|A|}=0
	\]
	and consequently
	\[  
		\sum_{A\in  M_{k1}}(-1)^{|A|}=\sum_{A\subset E_1}(-1)^{|A|}=(-1)^{k-1}.
	\]
	For $r>1$, the calculation of the sum
	\begin{align*}
		\sum_{A\in M_{kr}}(-1)^{|A|} &= \sum_{A_1\in 2^{E_1}\setminus \{E_1\}}
		\cdots\sum_{A_r\in 2^{E_r}\setminus \{E_r\}} (-1)^{|A_1\cup\cdots\cup A_r|} \\
		&= \sum_{A_1\in 2^{E_1}\setminus \{E_1\}}(-1)^{|A_1|} \cdots
		\sum_{A_r\in 2^{E_r}\setminus \{E_r\}}(-1)^{|A_r|} \\
		&= (-1)^{(k-1)r}
	\end{align*}
	yields the proof of the statement.
\end{proof}
\begin{lemma}\label{lemma:parity}
	Let $k,r$ be two positive integers and $\{E_1,\ldots,E_r\}$ a partition of a set $E$ 
	with exactly $r$ blocks of size $k$. Define a family of subsets of $E$ by
	\[ \mathcal{F}(k,r) =\{A\mid A\subseteq E\text{ and }
	\exists j\in \{1,\ldots,r \}:E_j\subseteq A \}. \]
	Then
	\[ \sum_{A\in \mathcal{F}(k,r)} (-1)^{|A|} = (-1)^{kr-r+1}. \]
\end{lemma}
\begin{proof}
	Define $E'=E\setminus E_r$,
	\begin{align*} 
		\mathcal{M}' &=\{A\cup E_r\mid A\subseteq E' \text{ and }
		\nexists i\in\{1,\ldots,r-1\}:E_i\subseteq A \},\\
		\mathcal{M}'' &= \{A\cup B \mid A\subseteq E',B\subseteq E_r \text{ and }
		\exists i\in\{1,\ldots,r-1\}:E_i\subseteq A \}.
	\end{align*}
	We can rewrite the definition of $\mathcal{M}'$ as
	\[ 
		\mathcal{M}'=\{A\cup E_r \mid A\in 
		2^{E_1}\setminus \{E_1\}\times\cdots\times 2^{E_{r-1}}\setminus \{E_{r-1}\} \},
	\]
	which shows that, according to Lemma \ref{lemma:pi}, each set of $\mathcal{M}'$ is 
	a disjoint union of a set of $\mathcal{M}_{k,r-1}$ and $E_r$. 
	By Lemma \ref{lemma:pi}, we obtain
	\begin{align*}
		\sum_{A\in \mathcal{M}'} (-1)^{|A|} 
		&= \sum_{A'\in \mathcal{M}_{k,r-1}} (-1)^{|A'|+|E_r|}=(-1)^{(k-1)(r-1)+k}
		=(-1)^{kr-r+1}.
	\end{align*}
	Observe that the set $\mathcal{F}(k,r)$ defined in Lemma \ref{lemma:pi} can be 
	represented as the disjoint union
	\[ 
		\mathcal{F}(k,r) = \mathcal{M}' \cup \mathcal{M}'',
	\]
	which implies
	\begin{align*} 
		\sum_{A\in \mathcal{F}(k,r)} (-1)^{|A|} 
		&= \sum_{A\in \mathcal{M}'} (-1)^{|A|} + \sum_{A\in \mathcal{M}''} (-1)^{|A|}. 
	\end{align*}
	In order to complete the proof, we show that the second sum vanishes. This follows from
	\[ 
		\mathcal{M}'' = 2^{E_r}\times \{A \mid A\subseteq E' \text{ and }
		\exists i\in\{1,\ldots,r-1\}:E_i\subseteq A \}
	\]
	and therefore
	\[ 
		\sum_{A\in \mathcal{M}''} (-1)^{|A|} = \sum_{A'\subseteq E_r} (-1)^{|A'|}
		\sum_{A''\in \mathcal{M}_{k,r-1}} (-1)^{|A''|}=0,
	\]
	since the first sum at the right-had side equals zero.
\end{proof}

\begin{theorem}\label{theo:complete}
	Let $G=(V,E)=K_{p,q}$ be a complete bipartite graph with $p+q$ vertices. Then
	\[ \sum_{F\subseteq E} (-1)^{|F|}N(G-F,x) = (-1)^{q-1}x^p + (-1)^{p-1}x^q.\]
\end{theorem}
\begin{proof}
	We use the presentation of the neighborhood polynomial as in the proof
	of Theorem \ref{theo:null}:
	\[ \sum_{F\subseteq E} (-1)^{|F|}N(G-F,x) 
	=  \sum_{W\in \mathcal{N}(G)} x^{|W|} \sum_{F\in \mathcal{E}(W)} (-1)^{m-|F|}  \]
	If $W$ is not a partition set of $K_{p,q}$ then there is again an edge that is not
	contained in $\partial W$, which can be used to show that the inner sum vanishes.
	Hence we can assume that $W$ is one of the two partition sets, say $|W|=p$. Then the
	minimal sets in $\mathcal{E}(W)$ are exactly $q$ disjoint sets of cardinality $p$ each.
	We observe that $\mathcal{E}(W)$ has exactly the structure of the set family
	$\mathcal{F}(k,r)$ employed in Lemma \ref{lemma:parity} with $k=p$ and $r=q$.
	From Lemma \ref{lemma:parity}, we obtain
	\[\sum_{A\in \mathcal{F}(p,q)} (-1)^{pq-|A|} 
	= (-1)^{pq-(pq-q+1)} = (-1)^{q-1} , \]
	which provides the term $(-1)^{q-1}x^p$ in the theorem, the second one is obtained
	in the same way.
\end{proof}

For the empty (edge-less) graph $G$, we have
\[ 
	\sum_{F\subseteq E} (-1)^{|F|}N(G-F,x) = N(G,x) = 1. 
\]
In the following statement, we use the notation $G\approx K_{p,q}$ to indicate that
$G$ is isomorphic to the disjoint union of $K_{p,q}$ and an empty graph.
\begin{theorem}\label{theo:N-poly}
	Define for any graph $G$
	\[ 
		h(G,x) = 
		\left\lbrace  
			\begin{array}{l}
				(-1)^{q+1}x^p + (-1)^{p+1}x^q, \text{ if }G\approx K_{p,q}, \\
				1, \text{ if }G\text{ is empty,} \\
				0, \text{ otherwise.}
			\end{array}
		\right. 
	\]
	The neighborhood polynomial of any graph $G$ satisfies
	\begin{align*} 
		N(G,x) &= \sum_{F\subseteq E} h((V,F),x) \\
			   &= 1+ \sum_{F\subseteq E:(V,F)\approx K_{p,q}}(-1)^{q+1}x^p + (-1)^{p+1}x^q.
	\end{align*}
\end{theorem}
\begin{proof}
	The statements of Theorem \ref{theo:null} and Theorem \ref{theo:complete} can be 
	combined to
	\[ 
		\sum_{F\subseteq E} (-1)^{m-|F|}N((V,F),x) = h(G,x).
	\]
	Now the statement follows by M\"{o}bius inversion.
\end{proof}

\section{Relations between Neighborhood and Domination Polynomials}

\begin{theorem}\label{theo:D-N}
	Let $G=(V,E)$ be a graph of order $n$ and $\bar{G}$ its complement, then
	\[ 
		D(G,x) + N(\bar{G},x) = (1+x)^n.
	\]
\end{theorem}
\begin{proof}
	The right-hand side of the equation is the ordinary generating function for all
	subsets of $V$. Therefore it suffices to show that for any graph $G=(V,E)$ the relation
	\[ 
		\mathcal{D}(G) \cup \mathcal{N}(\bar{G}) = 2^V
	\]
	is satisfied.
	Let $W\subseteq V$ be a non-dominating set of $G$. Then there exits a vertex $v\in V$
	with $N_G[v]\cap W=\emptyset $. Consequently $v$ is not adjacent to any vertex of $W$
	in $G$, which implies that $v$ is adjacent to each vertex of $W$ in $\bar{G}$. We
	conclude that $W\subseteq N_{\bar{G}}(v)$ and hence $W\in \mathcal{N}(\bar{G})$.
	
	Now let $W$ be a vertex set with $W\in \mathcal{N}(\bar{G})$. Then there is a vertex
	$v\in V$ with $W\subseteq N_{\bar{G}}(v)$, which implies that $N_G[v]\cap W=\emptyset $.
	We conclude that $W\notin \mathcal{D}(G)$. We have shown that any vertex subset of $V$
	belongs either to $\mathcal{D}(G)$ or to $\mathcal{N}(\bar{G})$, which completes the proof.
\end{proof}

\begin{theorem}
	Let $G$ be a simple graph of order $n$ such that its complement $\bar{G}$ has 
	exactly $a(G)$ subgraphs isomorphic to $K_{2p,2q}$ and exactly $b(G)$ subgraphs 
	isomorphic to $K_{2p+1,2q+1}$. Then
	\[
		d(G)=2^n -1 +2[a(G)-b(G)].
	\]
\end{theorem}
\begin{proof}
	By Theorem \ref{theo:D-N}, we obtain
	\[ 
		d(G) = D(G,1) = 2^n - N(\bar{G},1).
	\]
	The substitution of the last term according Theorem \ref{theo:N-poly} results in
	\[ 
		d(G) = 2^n - 1 - \sum_{F\subseteq E:(V,F)\approx K_{p,q}}(-1)^{q+1} + (-1)^{p+1}.
	\]
	We observe that the terms of the sum vanish when the parity of $p$ and $q$ differs.
	A term equals 2 if both $p$ and $q$ are odd, it equals -2 if both are even.
\end{proof}
\begin{corollary}[Brouwer, \cite{Brouwer2009}]
	The number of dominating sets of any finite graph is odd.
\end{corollary}

\printbibliography

\end{document}